\documentclass[lettersize,onecolumn]{IEEEtran}
\usepackage{amsmath,amsfonts,mathtools,mathrsfs, amsthm}
\usepackage{algorithmic}
\usepackage{algorithm}
\usepackage{array}
\usepackage[caption=false,font=normalsize,labelfont=sf,textfont=sf]{subfig}
\usepackage{textcomp}
\usepackage{stfloats}
\usepackage{url}
\usepackage{verbatim}
\usepackage{graphicx}
\usepackage{cite}
\usepackage{bm}
\usepackage{xcolor}
\usepackage{mathrsfs}
\usepackage{lineno}
\newcommand{\Simplex}{\Delta_{K-1}}
\newcommand{\SimplexDown}{\Delta_{K-1}^{\downarrow}}
\newcommand{\TM}{T_M}
\newcommand{\Hp}[1]{H^{(#1)}}

\newtheorem{theorem}{Theorem}[section]
\newtheorem{lemma}[theorem]{Lemma}
\newtheorem{proposition}[theorem]{Proposition}

\theoremstyle{definition}
\newtheorem{definition}[theorem]{Definition}
\newtheorem{remark}[theorem]{Remark}

\DeclareMathOperator{\Var}{Var}

\begin{document}

\title{Necessary and Sufficient Conditions for Characterizing Finite Discrete Distributions with Generalized Shannon's Entropy}

\author{Jialin Zhang
\thanks{Jialin Zhang is with the Department of Mathematics and Statistics, Mississippi State University, Mississippi State, MS 39762, USA (e-mail: jzhang@math.msstate.edu).}
}

\markboth{Preprint}%
{Shell \MakeLowercase{\textit{et al.}}: A Sample Article Using IEEEtran.cls for IEEE Journals}

\IEEEpubid{0000--0000~\copyright~2025 IEEE}

\maketitle

\begin{abstract}
This article establishes necessary and sufficient conditions under which a finite set of Generalized Shannon's Entropy (GSE) characterizes a finite discrete distribution up to permutation. For an alphabet of cardinality $K$, it is shown that $K\!-\!1$ distinct positive real orders of GSE are sufficient (and necessary if no multiplicity) to identify the distribution up to permutation. When the distribution has a known multiplicity structure with $s$ distinct values, $s\!-\!1$ orders are sufficient and necessary. These results provide a label‑invariant foundation for inference on unordered sample spaces and enable practical goodness‑of‑fit procedures across disparate alphabets. The findings also suggest new approaches for testing, estimation, and model comparison in settings where moment-based and link-based methods are inadequate.
\end{abstract}

\begin{IEEEkeywords}
distribution characterization, escort distribution, total positivity, Tchebycheff systems, P-matrix, Gale--Nikaid\^o univalence.
\end{IEEEkeywords}

\section{Introduction}
\label{sec:Intro}
\IEEEPARstart{S}{hannon's} entropy \cite{shannon1948mathematical} has been widely studied and applied across information theory, statistics and machine learning.
There is a vast literature on implications and variants of entropy \cite{kullback1951information, renyi1961measures, quinlan1986induction, zhang2010re, zhang2016entropic} and on entropy estimation \cite{madow1963maximum, harris1975statistical, paninski2003estimation, zhang2012entropy, zhang2012asymptotic, schurmann2015note}. Recent works on countably infinite alphabets studied the nonexistence of entropy under heavy tails and introduced the conditional distribution of total collision (CDOTC), a special escort distribution, which in turn led to Generalized Shannon's Entropy (GSE) \cite{zhang2020generalized, zhang2022entropic2}.
For any $m>0$, the CDOTC is in one-to-one correspondence with the original distribution; for orders at least two, GSE always exists and allows bypassing Lindeberg-type conditions in certain limit theorems \cite{zhang2022asymptotic}.
Follow-up studies showed that GSE can characterize discrete distributions even across disparate sample spaces \cite{zhang2024nonparametric} and can simplify asymptotic normality for dependence testing via generalized mutual information \cite{zhang2024normal}. This motivates a deeper exploration of GSE's characterization power.

Characterizing probability distributions is a foundational task across many disciplines \cite{kotz1974characterizations, galambos2006characterizations}. Classically, such characterizations often proceed via characteristic functions or moment-generating functions. However, an increasing number of modern sample spaces do not support moment-based concepts in a natural way because they lack an inherent order \cite{zhang2016statistical} —for example, genotype categories, biodiversity labels, or neurons in a neural network. In these settings, label-invariant information-theoretic functionals (e.g., entropies and functionals of entropies) provide order-free descriptors of distributions. Systematically developing distributional characterizations based on such information-theoretic quantities is therefore both natural and timely.

The existing result of characterization with GSE \cite{zhang2024nonparametric} requires a countably infinite set of GSE to uniquely\footnote{Throughout, ``uniquely'' means ``uniquely up to permutation'', since GSE is label-invariant.} determine a discrete distribution. This may be impractical to verify. One may therefore ask: How many GSE orders suffice if the distribution is known to be finite? Reducing from countably many to finitely many is of independent interest for practice (e.g., estimation stability, computational cost). Classical characterizations via finite sets of statistics often require integer-powered estimands of orders $1$ through $K$ \cite{zhang2016entropic, chen2019goodness}. These become statistically and computationally expensive as orders increase. In contrast, results in this article show that $K-1$ arbitrary positive real orders of GSE suffice for a $K$-discrete distribution, which may strengthen identifiability and offer practical gains.

The main contribution of this article is to establish the necessary and sufficient conditions for a set of GSE to characterize a finite discrete distribution. The results are briefly presented here:

\begin{itemize}
    \item[(i)] (\textbf{Sufficiency}) For $K\ge2$, any set of GSE from $r\ge K-1$ distinct positive real orders uniquely determines a cardinality-$K$ discrete distribution $p$.
    \item[(ii)] (\textbf{Necessity}) For $K\ge3$, if the cardinality-$K$ discrete distribution $p$ has no multiplicity, then no set of GSE from $r\le K-2$ distinct positive real orders uniquely determines $p$.
    \item[(iii)] (\textbf{Binary Case}) For $K=2$ and non-uniform ($i.e.$, $p_1 \neq p_2$), any positive single order of GSE suffices and is necessary to determine this distribution.
    \item[(iv)] (\textbf{Multiplicity}) If $p$ has $s\ge2$ distinct values, then a set of GSE from $r = s-1$ distinct positive real orders is sufficient and necessary to uniquely determine $p$.
\end{itemize}

The rest is organized as follows. Some preliminary definitions and notations are presented in Section \ref{sec:prelim}. Section \ref{sec:main_results} presents the statement of Theorem \ref{thm:main} and its proof. The theorem has four statements, with the proof of each statement in Sections \ref{subsec:proof-main(i)}, \ref{subsec:proof-main-(ii)}, \ref{subsec:proof-main(iii)}, and \ref{subsec:proof-main(iv)}. The major proof techniques for this article lie in Section \ref{subsec:proof-main(i)}, in which it is split into two further steps in Sections \ref{proof-III-i-step1} and \ref{proof-III-i-step2}. Finally, Section \ref{sec:conclusion} concludes the article with potential future works.

\section{Preliminaries, Notation, and Background}
\label{sec:prelim}
Let $Z$ be supported on $\mathscr{Z}=\{z_k: k=1,\ldots,K\}$ with an associated distribution $p=(p_1,\ldots,p_K)$, where $K<\infty$.
Define the simplex $\Simplex=\{p\in(0,1)^K:\sum_{i=1}^K p_i=1\}$ and its sorted simplex region $\SimplexDown=\{p\in\Simplex:p_1>\cdots>p_K\}$.

\begin{definition}[Conditional Distribution of Total Collision (CDOTC) and Generalized Shannon's Entropy (GSE) \cite{zhang2020generalized}]
For $m>0$, the Conditional Distribution of Total Collision (CDOTC) of $p$ with order $m$ is
\[
p^{(m)}=\Big\{p_k^{(m)}\Big\}=\left\{\frac{p_k^m}{\sum_{i=1}^K p_i^m}\right\}.
\]
Let
\[
\Hp{m}(p)=H\big(p^{(m)}\big) = - \sum_{i=1}^{K}p^{(m)}_i \ln p^{(m)}_i.
\]
$\Hp{m}(p)$ is the Generalized Shannon's Entropy (GSE) of $p$ with order $m$.    
\end{definition}
\begin{remark}
For any $m>0$, $p$ and its CDOTC ($i.e.$, $p^{(m)}$) uniquely determine each other \cite{zhang2020generalized}.    
\end{remark}
Write
\[
S_p(m):=\sum_{i=1}^K p_i^m,\quad y_p(m):=\ln S_p(m),\quad
\Hp{m}(p)=y_p(m)-m\,y_p'(m).
\]
Let
\[
w_i(m):=p_i^m/S_p(m), \quad \mu(m):=\sum_i w_i(m)\ln p_i=y_p'(m),
\]
and
\[
\alpha(m):=\mu(m)-\ln p_1\le0
\]
with
\[
\alpha'(m)=\Var_{w(m)}(\ln p_i)\ge0.
\]

\begin{definition}[Strictly total positivity (STP) of order $k$ \cite{karlin1966tchebycheff}] \label{def:stp}
Let $\mathcal{K}:S\times T\to\mathbb{R}$ be a kernel. Then $\mathcal{K}$ is
strictly totally positive of order $k$ (denoted $\mathrm{STP}_k$) if for
each $r=1,2,\dots,k$ and whenever $s_1<\cdots<s_r$ in $S$ and
$t_1<\cdots<t_r$ in $T$,
\[
\det\!\big[\,\mathcal{K}(s_i,t_j)\,\big]_{i,j=1}^{r}\;> 0 .
\]
\end{definition}

\begin{definition}[Strictly sign regularity (SSR) of order $k$ \cite{karlin1968total}] \label{def:ssr}
Let $\mathcal{K}:S\times T\to\mathbb{R}$ be a kernel. Then $\mathcal{K}$ is
strictly sign-regular of order $k$ (denoted $\mathrm{SSR}_k$) if there exists a sequence
$\{\varepsilon_r\}_{r=1}^k$ with each $\varepsilon_r\in\{+1,-1\}$ such that
for each $r=1,2,\dots,k$ and whenever $s_1<\cdots<s_r$ in $S$ and
$t_1<\cdots<t_r$ in $T$,
\[
\varepsilon_r\,\det\!\big[\,\mathcal{K}(s_i,t_j)\,\big]_{i,j=1}^{r}\;> 0 .
\]
\end{definition}

\begin{definition}[Extended Complete Tchebycheff (ECT) systems \cite{karlin1968total}] \label{def:ECT}
Let $\phi_1,\dots,\phi_n$ be real functions on an open interval $(a,b)$. Then $\{\phi_1,\dots,\phi_n\}$ is an extended complete Tchebycheff (ECT) system on $(a,b)$ if,
for each $r=1,2,\dots,n$,
\[
\det\!\big[\,\phi_j(x_i)\,\big]_{i,j=1}^{r} \;>\; 0
\qquad\text{for all } a\le x_1\le \cdots \le x_r\le b .
\]
\end{definition}

\begin{remark}
~
\begin{enumerate}
    \item A kernel $\mathcal{K}$ is STP if all square evaluation determinants with strictly increasing arguments are positive.
    \item A kernel $\mathcal{K}$ is SSR if all its square evaluation determinants are nonzero and have a constant sign.
    \item An ECT system is equivalent to the SSR of all evaluation determinants on the interval considered \cite[Ch.~6, Thm.~1.1]{karlin1966tchebycheff}.
\end{enumerate}
\end{remark}

\begin{definition}[Multiplicity]\label{def:stratum}
Let $K=n_1+\cdots+n_s$ with $s\in\{2,\ldots,K\}$.
Define the ordered stratum
\[
\Sigma^{\downarrow}_{\bm n}
=\Big\{\,p\in\SimplexDown:\ \exists a_1>\cdots>a_s>0 \text{ with } p \text{ consisting of } n_j \text{ copies of } a_j,\ \sum_{j=1}^s n_j a_j=1\,\Big\},
\]
which has intrinsic dimension $s-1$ in the chart $(a_2,\ldots,a_s)$ with $a_1=(1-\sum_{j\ge2}n_j a_j)/n_1$.
\end{definition}

\section{Main Results}
\label{sec:main_results}

\begin{theorem}[Finite-order GSE Characterization]\label{thm:main}
Let $K\ge 2$ and $\Simplex=\{p\in(0,1)^K:\sum_i p_i=1\}$. Let $M\subset(0,\infty)$ consist of $r$ distinct orders, and $T_M:\Simplex\to\mathbb{R}^r$ be the mapping that sends $p$ to $\{\Hp{m}(p):m\in M\}$.
\begin{itemize}
\item[(i)] ({Sufficient Condition}) If $r\ge K-1$, then $T_M$ is injective up to permutation.
\item[(ii)] ({Necessary Condition}) If $K\ge 3$ and $r\le K-2$ and $p$ has no multiplicity, then $T_M$ is not injective on $\Simplex$.
\item[(iii)] ({$K=2$}) For $K=2$ and non-uniform, a single positive order ($r=1$) is sufficient and necessary for $T_M$ to be injective.
\item[(iv)] ({Multiplicity Known}) If $p$ has $s$ distinct values, then $T_M$ is injective up to permutation for $r\ge s-1$; if $s\ge3$ and $r\le s-2$, then $T_M$ is not injective on $\Simplex$.
\end{itemize}
\end{theorem}

\subsection{Proof of Theorem \ref{thm:main} (i): Injectivity of $T_M$ for $r\ge K-1$}

\label{subsec:proof-main(i)}

The proof is carried out in two steps:

\begin{enumerate}
    \item Subsection \ref{proof-III-i-step1} proves that the Jacobian of $T_M$ is a P-matrix for all $p\in\SimplexDown$.
    \item Subsection \ref{proof-III-i-step2} proves that $T_M$ is globally injective.
\end{enumerate}

Consider $\SimplexDown$ with chart $(p_2,\ldots,p_K)$ and $p_1=1-\sum_{k=2}^K p_k$.

For $m>0$,
\begin{equation}\label{eq:grad}
\frac{\partial \Hp{m}}{\partial p_k}(p)
=\frac{m^2 p_k^{\,m-1}}{S_p(m)}\Big(\mu(m)-\ln p_k\Big),\quad k=1,\ldots,K,
\end{equation}
and hence $T_M$'s Jacobian, denoted as $DT_M (p)$, is
\begin{equation}\label{eq:colfactor}
DT_M (p) = \left(\frac{\partial \Hp{m}}{\partial p_k}-\frac{\partial \Hp{m}}{\partial p_1}\right)(p)
= m^2\,\frac{p_1^{\,m-1}}{S_p(m)}\,\Phi_m(u_k),\quad k=2,\ldots,K,
\end{equation}
where
\[
u_k:=\ln\frac{p_k}{p_1}\in(-\infty,0),\qquad
\Phi_m(u):=(\alpha(m)-u)e^{(m-1)u}-\alpha(m).
\]
Note that the sign of each element in $DT_M (p)$ is determined solely by $\Phi_m(u_k)$. Lemma \ref{lemma:LaplaceSTP} is stated next without proof since it is trivial based on Definition \ref{def:stp}.

\begin{lemma}\label{lemma:LaplaceSTP}
For $m,t>0$, Laplace Kernel $L(m,t)=e^{-mt}$ is STP on $(0,\infty)^2$. Consequently, for strictly increasing $\{m_i\},\{t_j\}$,
$\det[e^{-m_i t_j}]>0$.
\end{lemma}

\subsubsection{Proof of Step (1): $DT_M (p)$ is a P-matrix for all $p\in\SimplexDown$} \label{proof-III-i-step1}
~\\
Define $\mathcal{K}(m,u):=e^{(m-1)u}$ for $m>0$, $u<0$.
For fixed $p\in\SimplexDown$ and $C=\{c_1<\cdots<c_k\}\subset\{2,\ldots,K\}$ set $u_b:=u_{c_b}(p)$ and
\[
f_b(m):=\Phi_m(u_b)=(\alpha(m)-u_b)\,\mathcal{K}(m,u_b)-\alpha(m),\qquad b=1,\ldots,k.
\]

To begin, Lemma \ref{lemma:Tek} and Lemma \ref{lemma:KVD} are given with proof. Lemma \ref{lemma:Tek} proves the invariance of column-sign under diagonal flips, and Lemma \ref{lemma:KVD} is a transition lemma to elevate the results from two-point column mixing to a P-Matrix.

\begin{lemma} \label{lemma:Tek}
Let $T_k\in\mathbb{R}^{k\times k}$ be the unit lower–bidiagonal matrix
\[
T_k=\begin{pmatrix}
1      \\[-2pt]
1 & 1  \\
  & 1 & 1  \\
  &   & \ddots & \ddots \\
  &   &        & 1      & 1
\end{pmatrix},
\]
i.e., ones on the diagonal and the subdiagonal, zeros elsewhere. Let
$D=\mathrm{diag}(\delta_1,\dots,\delta_k)$ with $\delta_j\in\{\pm 1\}$.
Then
\begin{enumerate}
\item Every minor of $T_k$ is either $0$ or $1$; in particular, $T_k$ is totally positive.
\item For any matrix $A$, right-multiplication by $D$ only flips a fixed subset of column signs:
for any square index sets $I,J$ with $|I|=|J|$,
\[
\det\big((A D)_{I,J}\big)\;=\;\bigg(\prod_{j\in J} D_{jj}\bigg)\,\det\big(A_{I,J}\big).
\]
Thus, the sign of each minor of $A$ is changed by a column-dependent factor independent of the rows $I$. In particular, sign-regularity assertions are preserved.
\end{enumerate}
\end{lemma}

\begin{proof}[Proof of Lemma \ref{lemma:Tek}]
(1) Fix index sets $I=\{i_1<\cdots<i_r\}$ and $J=\{j_1<\cdots<j_r\}$ with $1\le r\le k$.
Column $j$ of $T_k$ has nonzeros only in rows $j$ and $j+1$, both equal to $1$.
Hence the submatrix $T_{k}[I,J]$ has the following sparsity: entry $(a,b)$ can be nonzero only if
$i_a\in\{j_b,\,j_b+1\}$. Consequently, after reordering rows increasingly, $T_{k}[I,J]$ is lower block–triangular with $0$–$1$ entries, and a nonzero determinant can occur only if the interlacing condition
\[
i_a\in\{j_a,\,j_a+1\}\qquad (a=1,\dots,r)
\]
holds. When this condition is met, there is exactly one permutation $\sigma\in S_r$ with all picked entries nonzero, namely $\sigma=\mathrm{id}$, and the corresponding product equals $1$. All other permutations vanish because they would require a column to contribute from a row on above $j_b$ or below $j_b+1$, which is zero. Therefore
\[
\det\big(T_{k}[I,J]\big)\in\{0,1\},
\]
and in particular every minor is nonnegative. Hence $T_k$ is totally positive \cite{karlin1968total}.

(2) Let $A\in\mathbb{R}^{m\times k}$ be arbitrary and $D=\mathrm{diag}(\delta_1,\dots,\delta_k)$ with $\delta_j\in\{\pm 1\}$.
Right-multiplying by $D$ multiplies column $j$ of $A$ by $\delta_j$.
For any square minor with column set $J$, one may factor out $\prod_{j\in J}\delta_j$, independent of the row choice $I$:
\[
\det\big((A D)_{I,J}\big)
= \det\Big(\big[\delta_j A_{\bullet j}\big]_{j\in J}\Big)
= \Big(\prod_{j\in J}\delta_j\Big)\det\big(A_{I,J}\big).
\]
Therefore right-multiplication by $D$ flips a predetermined set of column signs and cannot alter
statements asserting that all $k\times k$ minors share a common sign (up to that fixed signature).
\end{proof}

\begin{lemma} \label{lemma:KVD}
Let $\mathcal{K}(m, u)$ be an STP kernel on $(0,\infty)\times(-\infty,0)$. 
Fix $u_1<\cdots<u_k<0$. For $j\ge 2$ let $a_j(m),b_j(m)\ge0$ with $a_j$ nondecreasing and $b_j$ nonincreasing in $m$, and define the column family
\[
G_{\bullet,1}(m):=\mathcal{K}(m,u_1),\qquad
G_{\bullet,j}(m):=a_j(m)\,\mathcal{K}(m,u_j)+b_j(m)\,\mathcal{K}(m,u_{j-1})\quad(j\ge2).
\]
Then for every strictly increasing $m_1<\cdots<m_k$ the evaluation determinant 
$\det\,[G_{\bullet,j}(m_i)]_{i,j=1}^k$ has a constant nonzero sign. Equivalently, $\{G_{\bullet,1},\ldots,G_{\bullet,k}\}$ is strictly sign-regular (SSR) on $(0,\infty)$.
\end{lemma}

\begin{proof}[Proof of Lemma \ref{lemma:KVD}]

For each column $j$ and each row point $x=m$, regard
\[
G_{\bullet,j}(m)=\int \mathcal{K}(m,y)\,d\nu_j(m;y),
\qquad 
d\nu_j(m;y):=
\begin{cases}
\delta_{u_1}(dy), & j=1,\\[2pt]
a_j(m)\,\delta_{u_j}(dy)+b_j(m)\,\delta_{u_{j-1}}(dy), & j\ge2,
\end{cases}
\]
i.e. $G_{\bullet,j}$ is the image of a row-dependent discrete measure under the integral operator 
\[
(Tf)(x)\;=\;\int \mathcal{K}(x,y)\,f(y)\,d\mu(y),
\]
which is exactly Karlin's transform (3.3) in Chapter 5 of \cite{karlin1968total}. Here the underlying measure is a finite sum of Dirac masses whose weights depend monotonically on the row variable $x=m$. This places the setting under the hypotheses of Theorem 3.1 in Chapter 5 of \cite{karlin1968total}. 

Fix strictly increasing nodes $m_1<\cdots<m_k$.
By Karlin's composition identity for determinant transforms (determinant
form of Cauchy--Binet),
\begin{equation}\label{eq:KarlinComp}
\det\big[G_{\bullet,j}(m_i)\big]_{i,j=1}^k
=
\int \cdots \int
\det\big[\mathcal{K}(m_i,y_j)\big]_{i,j=1}^k
\;\prod_{j=1}^k d\nu_j(m_{j};y_j),
\end{equation}
where the column measures are
\[
d\nu_1(m;y)=\delta_{u_1}(dy),\qquad
d\nu_j(m;y)=a_j(m)\,\delta_{u_j}(dy)+b_j(m)\,\delta_{u_{j-1}}(dy)\quad(j\ge2),
\]
with $a_j(\cdot)\ge0$ nondecreasing and $b_j(\cdot)\ge0$ nonincreasing in $m$.
Thus, for each $j\ge2$, $d\nu_j$ is a two-point discrete measure supported on $\{u_{j-1},u_j\}$.

For a fixed column index $j\ge2$ and single integration variable $y_j$,
\[
\prod_{i=1}^k d\nu_j(m_i;y_j)
=
\prod_{i=1}^k\big[a_j(m_i)\,\delta_{u_j}(dy_j)+b_j(m_i)\,\delta_{u_{j-1}}(dy_j)\big].
\]
Any mixed term containing both $\delta_{u_j}(dy_j)$ and $\delta_{u_{j-1}}(dy_j)$ vanishes, hence
\[
\prod_{i=1}^k d\nu_j(m_i;y_j)
=
\Big(\prod_{i=1}^k a_j(m_i)\Big)\,\delta_{u_j}(dy_j)
\;+\;
\Big(\prod_{i=1}^k b_j(m_i)\Big)\,\delta_{u_{j-1}}(dy_j).
\]
Performing this independently for $j=2,\dots,k$ gives $2^{k-1}$ surviving choices.
Encode the choices by $\epsilon=(\epsilon_2,\ldots,\epsilon_k)\in\{0,1\}^{k-1}$ and define
\[
v_1^{(\epsilon)}:=u_1,\qquad 
v_j^{(\epsilon)}:=
\begin{cases}
u_j,& \epsilon_j=1,\\
u_{j-1},& \epsilon_j=0,
\end{cases}\quad (j\ge2),
\]
and row–column weights
\[
c_{ij}^{(\epsilon)}:=
\begin{cases}
a_j(m_i),& \epsilon_j=1,\\
b_j(m_i),& \epsilon_j=0,
\end{cases}
\qquad (i=1,\dots,k,\ j=2,\dots,k).
\]
Substituting these discrete expansions into \eqref{eq:KarlinComp} collapses the integrals to point evaluations and yields the finite sum
\begin{equation}\label{eq:finite-sum}
\det\big[G_{\bullet,j}(m_i)\big]_{i,j=1}^k
=
\sum_{\epsilon\in\{0,1\}^{k-1}}
\det\big[\mathcal{K}(m_i,\,v_j^{(\epsilon)})\big]_{i,j=1}^k
\;\prod_{i=1}^k\;\prod_{j=2}^k c_{ij}^{(\epsilon)}.
\end{equation}
Since $\mathcal{K}$ is strictly totally positive, any strictly increasing selection
$(v_1^{(\epsilon)},\ldots,v_k^{(\epsilon)})$ makes the kernel determinant in \eqref{eq:finite-sum} strictly positive, whereas selections with repeated $v_j^{(\epsilon)}$ yield zero. The goal is to show that for every strictly increasing $m_1<\cdots<m_k$ this determinant has
a constant positive sign (hence the column family is SSR).

The kernel $\mathcal{K}$ is STP, hence in particular $\mathrm{SSR}_r$ for every $r\ge2$.
The column construction is a rowwise transform of $\mathcal{K}$ using nonnegative, row–dependent
measures $d\nu_j(m;\cdot)$ whose weights $a_j(\cdot)$ are nondecreasing and $b_j(\cdot)$ are nonincreasing. By Karlin’s variation–diminishing theorem (Theorem 3.1(i)–(ii) in Chapter 5 of \cite{karlin1968total}), such transforms do not increase the relevant sign-change count; in the $\mathrm{SSR}_r$ case, inequality (3.7) in Chapter 5 of \cite{karlin1968total} applies. Consequently, for any fixed order $k$ and strictly increasing $m_1<\cdots<m_k$, the $k\times k$ evaluation determinants of the transformed columns have a common sign (no sign flips as the $m$’s vary within the domain).

Karlin’s converse (Theorem 3.1(iii)–(iv), Chapter 5 of \cite{karlin1968total}) implies that if the variation–diminishing inequality holds for the transform, then the induced set-kernel $\mathcal{K}(x,E):=\int_E \mathcal{K}(x,y)\,d\mu(y)$ inherits $\mathrm{SR}_r$ or $\mathrm{SSR}_r$ (under the mild cardinality hypotheses listed there). For the discrete two-point measures used here, this yields that the column family $\{G_{\bullet,1},\ldots,G_{\bullet,k}\}$ obtained from $\mathcal{K}$ by the two-point mixing is $\mathrm{SSR}_k$: all $k\times k$ evaluation determinants share the same nonzero sign.

Each summand in \eqref{eq:finite-sum} is the product of a kernel determinant and a nonnegative weight. Because $\mathcal{K}$ is STP, the determinant $\det[\mathcal{K}(m_i, v_j^{(\epsilon)})]$ is strictly positive when the selected $y$–nodes $(v_1^{(\epsilon)},\ldots,v_k^{(\epsilon)})$ are strictly increasing, and is zero if some nodes repeat. Hence, the entire sum is nonnegative. Furthermore, row nontriviality
(for each $i$ and $j\ge2$, at least one of $a_j(m_i), b_j(m_i)$ is positive) guarantees that a strictly increasing $y$–pattern occurs with a strictly positive weight; its STP determinant is then strictly positive. Therefore, the $k$th-order evaluation determinant is strictly positive. This is similar to the “strict” part in Theorem 3.1 (ii) / (iv) and is summarized again in Theorem 3.2 in Chapter 5 of \cite{karlin1968total}.

Therefore, for every strictly increasing $m_1<\cdots<m_k$,
\[
\det\,[G_{\bullet,j}(m_i)]_{i,j=1}^k \;>\; 0.
\]
It follows that all $k\times k$ evaluation determinants have a constant nonzero sign, $i.e.$,
the column family $\{G_{\bullet,1},\ldots,G_{\bullet,k}\}$ is SSR on $(0,\infty)$.
\end{proof}

Next, Proposition \ref{prop:ECT-bridge} provides an ECT bridge for the core $\Phi$-columns, which is the final piece of support needed for the proof of this step.

\begin{proposition} \label{prop:ECT-bridge}
Fix $p\in\Delta^{\downarrow}_{K-1}$ and set $u_2<\cdots<u_K<0$ with $u_j=\ln(p_j/p_1)$ for $j=2,\ldots,K$. For $m>0$ define
\[
f_j(m):=\Phi_m(u_{j+1})=(\alpha(m)-u_{j+1})\,e^{(m-1)u_{j+1}}-\alpha(m),\qquad j=1,\ldots,K-1.
\]
Then for each $k\in\{1,\ldots,K-1\}$ and any strictly increasing $m_1<\cdots<m_k$,
\[
\det\big[f_b(m_a)\big]_{a,b=1}^{k}>0.
\]
Equivalently, $\{f_1,\ldots,f_k\}$ forms an ECT system on $(0,\infty)$.
\end{proposition}

\begin{proof}
Let $\mathcal{K}(m,u):=e^{(m-1)u}$ for $m>0$, $u<0$. By Lemma~\ref{lemma:LaplaceSTP}, $\mathcal{K}$ is STP on $(0,\infty)\times(-\infty,0)$; hence, for any strictly increasing
$m_1<\cdots<m_k$ and $v_1<\cdots<v_k<0$,
\[
\det\big[\mathcal{K}(m_a,v_b)\big]_{a,b=1}^{k}>0.
\]
Therefore the column family $\{\mathcal{K}(\,\cdot\,,u_2),\ldots,\mathcal{K}(\,\cdot\,,u_{K})\}$ is an ECT system on $(0,\infty)$.

For fixed $k$, write
\[
F:=\big[f_b(m_a)\big]_{a,b=1}^{k}
\]
with $f_b(m)=\Phi_m(u_{b+1})$, and introduce the forward–difference matrix
\[
S_k=\begin{pmatrix}
1 & -1 &        &        &   \\
  &  1 &  -1    &        &   \\
  &    & \ddots & \ddots &   \\
  &    &        &   1    & -1\\
  &    &        &        &  1
\end{pmatrix},\qquad \det(S_k)=1 .
\]
Right–multiplication by $S_k$ performs columnwise differences:
\[
G := F S_k = \big[g_1,\ldots,g_k\big],
\quad g_1=f_1,\qquad
g_j=f_j-f_{j-1}\ \ (j\ge2).
\]
Hence $\det(F)=\det(G)$.

A direct computation gives, for $j\ge2$,
\[
g_j(m)=f_j(m)-f_{j-1}(m)
=\big(\alpha(m)-u_{j+1}\big)\mathcal{K}(m,u_{j+1})
-\big(\alpha(m)-u_{j}\big)\mathcal{K}(m,u_{j}),
\]
and $g_1(m)=f_1(m)=(\alpha(m)-u_2)\mathcal{K}(m,u_2)-\alpha(m)$.
Recall that $\alpha'(m)=\mathrm{Var}_{w(m)}(\ln p_i)\ge 0$, so both
$m\mapsto \alpha(m)-u_{j+1}$ and $m\mapsto \alpha(m)-u_j$ are nondecreasing.

Fix $s\in\{2,\ldots,K-1\}$ and consider the open interval
\[
I_s := \big\{\,m>0:\ u_s<\alpha(m)<u_{s+1}\,\big\},
\]
together with the boundary points where $\alpha(m)=u_r$ for some $r$.
On $I_s$ the signs of $\alpha(m)-u_j$ are fixed for each $j$;
define a column–sign matrix $D_s=\mathrm{diag}(\delta_1,\ldots,\delta_k)$ (independent of $m$) by
\[
\delta_j :=
\begin{cases}
+1, & j\le s-1,\\
-1, & j\ge s,     
\end{cases}
\]
and set $\widetilde G := G D_s=\big[\tilde g_1,\ldots,\tilde g_k\big]$.
By Lemma~\ref{lemma:Tek}(2), right–multiplication by a diagonal $\{\pm1\}$ matrix only flips a fixed subset
of column signs and thus preserves any sign–regularity assertion.

For $j\ge2$, one then have on $I_s$ the adjacency–mixing form
\[
\tilde g_j(m)= a_j^{(s)}(m)\,\mathcal{K}(m,u_{j+1}) + b_j^{(s)}(m)\,\mathcal{K}(m,u_{j}),
\]
where
\[
a_j^{(s)}(m)=
\begin{cases}
\alpha(m)-u_{j+1}\ \ (\ge0), & j\le s-1,\\
u_{j+1}-\alpha(m)\ \ (\ge0), & j\ge s,
\end{cases}
\qquad
b_j^{(s)}(m)=
\begin{cases}
-(\alpha(m)-u_{j})=u_j-\alpha(m)\ \ (\ge0), & j\le s-1,\\
-(u_j-\alpha(m))=\alpha(m)-u_j\ \ (\ge0), & j\ge s.
\end{cases}
\]
Because $\alpha$ is nondecreasing, $a_j^{(s)}(\cdot)$ is nondecreasing and
$b_j^{(s)}(\cdot)$ is nonincreasing on $I_s$.
Thus, for $j\ge2$, each $\tilde g_j$ is a two–point adjacent mixture of the Laplace columns
$\mathcal{K}(\cdot,u_{j})$ and $\mathcal{K}(\cdot,u_{j+1})$ with nonnegative row–dependent weights
having the required monotonicity.

The first column $\tilde g_1=\delta_1 g_1$ has the form
\[
\tilde g_1(m)= c_1^{(s)}(m)\,\mathcal{K}(m,u_2) + d_1^{(s)}(m)\,\mathcal{K}(m,0),
\]
with $c_1^{(s)}(m)=|\alpha(m)-u_2|$ and $d_1^{(s)}(m)=|\alpha(m)|$, since $\mathcal{K}(m,0)=1$.
On each $I_s$, the functions $c_1^{(s)}, d_1^{(s)}$ are nonnegative and piecewise monotone with
the same one–sided properties as above (one nondecreasing, the other nonincreasing).

By Lemma~\ref{lemma:LaplaceSTP}, the kernel $\mathcal{K}(m,u)=e^{(m-1)u}$ is STP on $(0,\infty)\times(-\infty,0]$.
Therefore, Lemma~\ref{lemma:KVD} applies on each $I_s$ to the
column family $\{\tilde g_1,\ldots,\tilde g_k\}$:
for any strictly increasing $m_1<\cdots<m_k$ in $I_s$, the $k\times k$ evaluation determinants
have a common nonzero sign, i.e., $\{\tilde g_1,\ldots,\tilde g_k\}$ is $\mathrm{SSR}_k$ on $I_s$.
By Lemma~\ref{lemma:Tek}(2) the same SSR conclusion holds for $G$ on $I_s$.

At a boundary point $m_\star$ with $\alpha(m_\star)=u_r$,
each $\tilde g_j(m_\star)$ collapses to either $\mathcal{K}(m_\star,u_{j})$ or $\mathcal{K}(m_\star,u_{j+1})$
(up to a positive scalar), while $\tilde g_1(m_\star)$ becomes a positive scalar multiple of
either $\mathcal{K}(m_\star,u_2)$ or $\mathcal{K}(m_\star,0)$.
Thus the evaluation matrix at such boundary nodes reduces to a submatrix of
$\big[\mathcal{K}(m_a, v_b)\big]$ with strictly increasing $u$–nodes $v_b\in\{u_2,\ldots,u_{k+1},0\}$,
whose determinant is strictly positive by STP (Lemma~\ref{lemma:LaplaceSTP}).
Since determinants depend continuously on the entries and the sign on each $I_s$
is constant (SSR), the same strict positivity holds throughout each $I_s$.
Patching the intervals together yields strict positivity for all $m_1<\cdots<m_k$ in $(0,\infty)$.

Because $\det(F)=\det(G)$ and $G$ has strictly positive $k\times k$ evaluation determinants, one may conclude

\[
\det\big[f_b(m_a)\big]_{a,b=1}^k>0
\quad\text{for all strictly increasing } m_1<\cdots<m_k.
\]
Equivalently, $\{f_1,\ldots,f_k\}$ forms an ECT system on $(0,\infty)$.
\end{proof}

Finally, with all the previous results in this step, Proposition \ref{prop:globalP} concludes the proof in this step.

\begin{proposition}\label{prop:globalP}
$D T_M(p)$ is a P-matrix on $\SimplexDown$.
\end{proposition}

\begin{proof}[Proof of Proposition \ref{prop:globalP}]
Fix index sets $R=\{j_1<\cdots<j_k\}\subset\{1,\ldots,r\}$ and 
$C=\{c_1<\cdots<c_k\}\subset\{2,\ldots,K\}$. Define the core evaluation determinant
\[
\Delta_{R,C}(p):=\det\!\big[\Phi_{m_{j_a}}(u_{c_b}(p))\big]_{a,b=1}^k .
\]

By Proposition~\ref{prop:ECT-bridge}, for any strictly increasing nodes 
$m_{j_1}<\cdots<m_{j_k}$:
\begin{equation}\label{eq:Delta-positive}
\Delta_{R,C}(p)>0\qquad\text{for all }p\in\SimplexDown.
\end{equation}

By the column–factorization identity \eqref{eq:colfactor}, each 
$k\times k$ principal minor of $D T_M(p)$ with row set $R$ and column set $C$ equals
$\Delta_{R,C}(p)$ multiplied by a strictly positive prefactor (depending on $p$ but not on the sign).
Combining this with \eqref{eq:Delta-positive} shows that every principal minor of 
$D T_M(p)$ is strictly positive on $\SimplexDown$.

Since all principal minors are positive, $D T_M(p)$ is a P-matrix on $\SimplexDown$.
\end{proof}

\subsubsection{Proof of Step (2): $T_M$ is globally injective on $\Delta_{K-1}^{\downarrow}$ for $r\ge K-1$} \label{proof-III-i-step2}
~\\
Gale and Nikaid\^o's univalence theorem (Theorem 4 of \cite{gale1965jacobian}) works if and only if $\Delta_{K-1}^{\downarrow}$ is a rectangular region of $\mathbb{R}^{K-1}$. However, $\Delta_{K-1}^{\downarrow}$ is a convex region of $\mathbb{R}^{K-1}$ under its definition. A Gale--Nikaid\^o univalence theorem on a convex region is needed and it is presented with proof in Proposition \ref{prop:P-on-convex}.

\begin{proposition} \label{prop:P-on-convex}
    Let $\Omega\subset\mathbb{R}^n$ be convex and open. If a differentiable map $F:\Omega\to\mathbb{R}^n$ has a P-matrix Jacobian $DF(x)$ for all $x\in\Omega$, then $F$ is injective on $\Omega$.
\end{proposition}
\begin{proof}[Proof of Proposition \ref{prop:P-on-convex}]
Fix distinct $a,b\in\Omega$ and set $v:=b-a\neq0$. Let $\gamma(t):=a+t\,v$ for $t\in[0,1]$. For each $i$ define
\[
g_i(t):=v_i\big(F_i(\gamma(t))-F_i(a)\big),\qquad 
\psi(t):=\max_{1\le i\le n} g_i(t),
\]
and the active index set $I(t):=\{\,i:\ g_i(t)=\psi(t)\,\}$. By differentiability of $F$ and the chain rule,
\[
g_i'(t)=v_i\big(DF(\gamma(t))\,v\big)_i \quad (t\in[0,1]).
\]
Recall the Fiedler–Pt\'ak criterion (see Theorem 3.3 (ii) of \cite{fiedler1962matrices}): for every nonzero $w\in\mathbb{R}^n$,
\[
\max_{1\le i\le n} w_i\,(A w)_i>0\qquad\text{whenever}\ A\ \text{is a $P$-matrix}.
\]
Since $DF(x)$ is a $P$-matrix for all $x\in\Omega$, for every $t\in[0,1]$:
\begin{equation}\label{eq:FP-along-line}
\max_{1\le i\le n} g_i'(t)=\max_{1\le i\le n} v_i\big(DF(\gamma(t))\,v\big)_i \;>\;0.
\end{equation}

Next is to show that $\psi$ is strictly increasing on $[0,1]$.

The upper right Dini derivative of the pointwise maximum satisfies
\[
D^+\psi(t):=\limsup_{\delta\downarrow0}\frac{\psi(t+\delta)-\psi(t)}{\delta}
\ \ge\ \max_{i\in I(t)} g_i'(t)\qquad (t\in[0,1)).
\]
This follows from $\psi(t+\delta)\ge g_i(t+\delta)$ for every $i$ and taking $\limsup_{\delta\downarrow0}$.

Fix $t_0\in[0,1)$. Suppose, to the contrary, that $\psi(t)\le\psi(t_0)$ for all $t\in[t_0,t_0+\varepsilon]$ with some $\varepsilon>0$.
Then for any $i$ and $0<\delta\le\varepsilon$,
\[
\frac{g_i(t_0+\delta)-g_i(t_0)}{\delta}
\ \le\ \frac{\psi(t_0+\delta)-\psi(t_0)}{\delta}\ \le\ 0,
\]
hence $g_i'(t_0)\le 0$. This contradicts \eqref{eq:FP-along-line}, which ensures that
$\max_i g_i'(t_0)>0$. Therefore for every $t_0\in[0,1)$ there exists $\varepsilon(t_0)>0$ with
\[
\psi(t_0+\varepsilon(t_0))>\psi(t_0).
\]

Define $S:=\{t\in[0,1]:\ \psi(t)>\psi(0)\}$. By Step 2, $S$ is nonempty and open in $[0,1]$.
Let $s:=\sup S$. If $s<1$, applying Step 2 at $t_0=s$ yields a small $\varepsilon>0$ with
$\psi(s+\varepsilon)>\psi(s)$, contradicting the definition of $s$. Hence $s=1$, i.e.,
$\psi(1)>\psi(0)=0$. Thus $F(b)\neq F(a)$, and $F$ is injective on $\Omega$.
\end{proof}

By Proposition~\ref{prop:globalP} and Proposition~\ref{prop:P-on-convex}, $T_M:\Delta_{K-1}^{\downarrow}\to\mathbb{R}^r$ is injective for $r \geq K-1$ since $\Delta_{K-1}^{\downarrow}$ is convex.

\begin{remark}
The inequality $r\ge K-1$ is used only in Step (3), where a square Jacobian is required to invoke Proposition \ref{prop:P-on-convex}. If $r=K-1$, the Jacobian $D T_M(p)$ is square and a P-matrix. If $r>K-1$, choose any row set $R$ with $|R|=K-1$ and define the submap $F_R:=\pi_R\circ T_M:\Delta_{K-1}^{\downarrow}\to\mathbb{R}^{K-1}$ by keeping the coordinates in $R$. By Steps (1)–(2), the Jacobian $D F_R(p)$ is a P-matrix for all $p$, hence $F_R$ is injective by Proposition~\ref{prop:P-on-convex}. In contrast, for $r<K-1$, the injectivity on a $(K-1)$-dimensional convex domain may not hold (see Section \ref{subsec:proof-main-(ii)}).
\end{remark}

\subsection{Proof of Theorem \ref{thm:main}(ii): $T_M$ is not injective for $r<K-1$ when $p$ has no multiplicity}
\label{subsec:proof-main-(ii)}
\begin{proof}[Proof of Theorem \ref{thm:main}(ii)]

Consider the skew ray
\[
p(\varepsilon)=\big(1-\sum_{k=2}^K \varepsilon_k,\ \varepsilon_2,\ldots,\varepsilon_K\big),
\qquad 0<\varepsilon_K\ll\cdots\ll\varepsilon_2\ll1,
\]
which lies in $\Delta_{K-1}^{\downarrow}$ and has strictly ordered coordinates (no multiplicity).
Along this ray,
\[
S_{p}(m)=p_1^m(1+o(1)),\qquad \alpha(m)=o(1)\quad\text{as }\varepsilon\to0,
\]
uniformly for $m\in M$ (finite). Using the column factorization \eqref{eq:colfactor},
each entry of the Jacobian $D T_M(p)$ equals a positive prefactor times
\[
\Phi_m(u_k)=(\alpha(m)-u_k) e^{(m-1)u_k}-\alpha(m)
=(-u_k)e^{(m-1)u_k}(1+o(1)).
\]
Let $\rho_k:=p_k/p_1\in(0,1)$ and $t_k:=-\ln\rho_k=-u_k>0$; for the skew ray $t_2<\cdots<t_K$.
Fix the $r\times r$ block with row indices $m_1,\dots,m_r$ and column indices $k=2,\dots,r+1$.
Up to positive row/column scalings and an $(1+o(1))$ factor, this block reduces to
\[
B_{j\ell}\;=\;e^{-\,m_j\,t_{\,\ell+1}}\qquad (j,\ell=1,\dots,r).
\]
By Lemma~\ref{lemma:LaplaceSTP}, the Laplace kernel $L(m,t)\mapsto e^{-mt}$ is STP on $(0,\infty)^2$,
hence for strictly increasing $\{m_j\}$ and $\{t_{\ell+1}\}$,
\[
\det(B)\;=\;\det\!\big[e^{-m_j t_{\ell+1}}\big]_{j,\ell=1}^r\;>\;0.
\]
Therefore, for all sufficiently small $\varepsilon$, the corresponding $r\times r$ block of
$D T_M\big(p(\varepsilon)\big)$ is nonsingular, implying
\[
\operatorname{rank}\big(D T_M\big(p(\varepsilon)\big)\big)\;=\;r.
\]
Choose such a point $\varepsilon^\star$ and denote the corresponding ray by $p^\star:=p(\varepsilon^\star)$.

Since the determinant of the chosen $r\times r$ minor is continuous in $p$ and nonzero at $p^\star$,
it remains nonzero on a neighborhood $U$ of $p^\star$ in $\Delta_{K-1}^{\downarrow}$.
Hence $T_M$ has constant rank $r$ on $U$. By Rank Theorem (Theorem 4.12 of \cite{lee2003introduction}), there exists a $y\in \mathbb{R}^r$ near $T_M(p^\star)$ such that $T_M^{-1}(y)\cap U$ is a real-analytic submanifold of dimension
\[
\dim\big(\Delta_{K-1}^{\downarrow}\big)-r\;=\;(K-1)-r\;\ge\;1.
\]
This submanifold contains infinitely many distinct elements (probability distributions) in $\Delta_{K-1}^{\downarrow}$ that map to the same point in $\mathbb{R}^r$. This shows that $T_M$ is not injective on $\Delta_{K-1}^{\downarrow}$ whenever $r<K-1$.
\end{proof}

\subsection{Proof of Theorem \ref{thm:main}(iii): The case $K=2$}
\label{subsec:proof-main(iii)}
\begin{proof}[Proof of Theorem \ref{thm:main}(iii)]
WLOG assume $p_1 > p_2$.
Let $p:=p_1$ and $w:=\dfrac{p^m}{p^m+(1-p)^m}\in(1/2,1)$.
Then
\[
\frac{dH}{dw}=\ln\frac{1-w}{w},\qquad
\frac{dw}{dp}=\frac{m\,p^{m-1}(1-p)^{m-1}}{(p^m+(1-p)^m)^2}>0,
\]
so
\[
\frac{d\,\Hp{m}(p)}{dp}
=\frac{dH}{dw}\cdot\frac{dw}{dp}
=\ln\frac{1-w}{w}\cdot\frac{m\,p^{m-1}(1-p)^{m-1}}{(p^m+(1-p)^m)^2}<0.
\]
Strict monotonicity implies injectivity.
\end{proof}

\subsection{Proof of Theorem \ref{thm:main}(iv): Multiplicity known}
\label{subsec:proof-main(iv)}

\subsubsection{Necessity}
\begin{proposition}\label{prop:stratum-lb}
If $r < s-1$, then $\TM$ is not injective on $\Sigma^{\downarrow}_{\bm n}$.
\end{proposition}
\begin{proof}[Proof of Proposition \ref{prop:stratum-lb}]
Similar to the proof of Theorem \ref{thm:main}(ii) in Subsection \ref{subsec:proof-main-(ii)}.
\end{proof}

\subsubsection{Sufficiency}

\begin{proposition}\label{prop:s-globalP}
$D T_M(p)$ is a P-matrix on $\Sigma^{\downarrow}_{\bm n}$.
\end{proposition}
\begin{proof}[Proof of Proposition \ref{prop:s-globalP}]
Similar to the proof of Proposition \ref{prop:globalP} in Subsection \ref{proof-III-i-step1}.
\end{proof}

By Proposition~\ref{prop:s-globalP} and Proposition~\ref{prop:P-on-convex}, $T_M:\Sigma^{\downarrow}_{\bm n}\to\mathbb{R}^r$ is injective for $r \geq s-1$ since $\Sigma^{\downarrow}_{\bm n}$ is convex.

\section{Conclusion}

\label{sec:conclusion}

This article has established necessary and sufficient conditions under which Generalized Shannon's Entropy (GSE) characterizes a finite discrete distribution. The results sharpen the theoretical understanding of GSE and broaden its utility for distributional characterization. From a practical perspective, GSE immediately enables goodness-of-fit procedures for comparing probability distributions, including comparisons across disparate sample spaces where link-based tools are ill-suited ($e.g.$, Pearson's Chi-squared Goodness-of-fit Test \cite{pearson1900x}).

Several directions merit further investigation. One avenue is to study data- and task-dependent choices of GSE orders, seeking optimal selections under different scenarios. Another is to extend the present discrete framework to continuous settings, clarifying measure-theoretic requirements and stability properties. These developments would deepen both the theoretical and applied reach of GSE-based methods.

\bibliographystyle{IEEEtran}
\bibliography{ref.bib}

\end{document}